\documentclass
{elsarticle}

\usepackage{amssymb,amsmath,amsthm,enumerate,bbm}
\usepackage{graphicx}
\usepackage{verbatim}

\usepackage{color}

\usepackage[
hypertexnames=false, colorlinks, citecolor=red, linkcolor=red]{hyperref} 
\hypersetup{bookmarksdepth=3}

\DeclareMathOperator{\Ker}{Ker}

\DeclareMathOperator{\Tr}{Tr}

\DeclareMathOperator{\Span}{span}

\newcommand{\abs}[1]{\lvert#1\rvert}
\newcommand{\Abs}[1]{\left\lvert#1\right\rvert}
\newcommand{\norm}[1]{\lVert#1\rVert}

\newcommand{\jap}[1]{\langle#1\rangle}

\newcommand{\bbN}{{\mathbb N}}
\newcommand{\bbR}{{\mathbb R}}
\newcommand{\bbC}{{\mathbb C}}

\newcommand{\calC}{\mathcal{C}}

\newcommand{\1}{\mathbf{1}}

\newcommand{\dd}{\mathrm d}

\numberwithin{equation}{section}

\theoremstyle{plain}
\newtheorem{theorem}{\bf Theorem}[section]
\newtheorem*{theorem*}{Theorem}
\newtheorem{lemma}[theorem]{\bf Lemma}

\newtheorem*{proposition*}{\bf Proposition}

\theoremstyle{definition}

\newtheorem*{definition*}{\bf Definition}

\theoremstyle{remark}
\newtheorem*{remark*}{\bf Remark}

\newtheorem*{example*}{\bf Example}

\makeatletter
\newcommand*\bigcdot{\mathpalette\bigcdot@{.5}}
\newcommand*\bigcdot@[2]{\mathbin{\vcenter{\hbox{\scalebox{#2}{\,\,$\m@th#1\bullet$\,\,}}}}}
\makeatother

\begin{document}

\begin{frontmatter}

\title{The spectral map for weighted Cauchy matrices is an involution}

\author[1]{Alexander Pushnitski}
\ead{alexander.pushnitski@kcl.ac.uk}

\author[2]{Sergei Treil \fnref{fn2}}
\ead{treil@math.brown.edu}

\affiliation[1]{organization={Department of Mathematics,  Kings College London},
address={Strand,},
city={London},
postcode={WC2R~2LS},
country={U.K.}
}

\affiliation[2]{organization={Department of Mathematics, Brown University},
city={Providence},
state={RI},
postcode={02912},
country={USA}
}

\fntext[fn2]{supported  in part by the National Science Foundation under the grant   
DMS-2154321}

\date{\today}

\begin{abstract}
Let $N$ be a natural number. We consider weighted Cauchy matrices of the form 
\[
\mathcal{C}_{a,A}=\left\{\frac{\sqrt{A_j A_k}}{a_k+a_j}\right\}_{j,k=1}^N,
\]
where $A_1,\dots,A_N$ are positive real numbers and $a_1,\dots,a_N$ are distinct positive real 
numbers, listed in increasing order. Let $b_1,\dots,b_N$ be the eigenvalues of $\mathcal{C}_{a,A}$, 
listed in increasing order. Let $B_k$ be positive real numbers such that $\sqrt{B_k}$ is the 
Euclidean norm of the orthogonal projection of the vector 
\[
v_A=(\sqrt{A_1},\dots,\sqrt{A_N})
\]
onto the $k$'th eigenspace of $\mathcal{C}_{a,A}$. We prove that the spectral map $(a,A)\mapsto 
(b,B)$ is an involution and discuss simple properties of this map. 
\end{abstract}


\end{frontmatter}

\section{Introduction and main result}\label{sec.z}
Let $N\in\bbN$. 
\begin{enumerate}[(i)]
\item
Let 
$a=(a_1,\dots, a_N)$, where $a_k$ are \emph{distinct} positive real numbers listed in increasing 
order:
\[
a_1<a_2<\dots<a_N.
\]
\item
Let $A=(A_1,\dots,A_N)$, where $A_k$ are positive real numbers. 
\end{enumerate}
For $a$ and $A$ be as above, we consider the \emph{weighted Cauchy matrix} of the form
\begin{equation}
\calC_{a,A}=\left\{\frac{\sqrt{A_j A_k}}{a_k+a_j}\right\}_{j,k=1}^N.
\label{eq:17}
\end{equation}
It is well known (and we will prove it below) that $\calC_{a,A}$ is positive definite.
Furthermore, we will see that
\begin{equation}
v_A=(\sqrt{A_1},\dots,\sqrt{A_N})
\label{eq:19}
\end{equation}
is a cyclic vector for $\calC_{a,A}$, and so $\calC_{a,A}$ has $N$ positive distinct eigenvalues. 
Let us write these eigenvalues in increasing order as $b=(b_1,\dots, b_N)$, and let $f_1,\dots,f_N$ 
be the corresponding normalised (with respect to the standard Euclidean norm in $\bbC^N$) 
eigenvectors. 

Let us define 
\begin{equation}
B_k=\abs{\jap{v_A,f_k}}^2,
\label{eq:20}
\end{equation}
where $\jap{\cdot,\cdot}$ is the standard Euclidean inner product in $\bbC^N$. In other words, 
$B_k$ are positive numbers such that $\sqrt{B_k}$ is the norm of the projection of $v_A$ onto the 
$k$'th eigenspace of $\calC_{a,A}$. 

We have thus defined the \emph{spectral map} $\Omega: Z_N\to Z_N$
\[
\Omega: (a,A)\mapsto (b,B),
\]
where $B=(B_1,\dots,B_N)$, and  $Z_N$ is the set of all pairs $(a,A)$ satisfying the above 
conditions (i) and (ii). Our main result is:

\begin{theorem}\label{thm.4}
The map $\Omega$ is an involution on $Z_N$. In particular, it is a bijection on $Z_N$. 
\end{theorem}
This includes three non-trivial statements:
\begin{itemize}
\item
The map $\Omega$ is an injection, i.e. the pair $(b,B)$ uniquely determines the matrix 
$\calC_{a,A}$. 
\item
The map $\Omega$ is a surjection, i.e. any pair $(b,B)$ satisfying (i),(ii) corresponds to some 
matrix $\calC_{a,A}$. 
\item
In order to reconstruct the pair $(a,A)$ (i.e.~the Cauchy matrix $\calC_{a,A}$) from $(b,B)$, one 
needs to apply the map $\Omega$ to $(b,B)$. 
\end{itemize}
The proof of Theorem~\ref{thm.4} is given in Section~\ref{sec.e}. 

Let us also discuss some properties of the map $\Omega$. To give some intuition, we note that for 
$N=1$ the spectral map is immediate to compute:
\[
\Omega(a,A)=\left(\frac{A}{2a},A\right), \quad N=1.
\]
First we list some simple scaling properties of $\Omega$. 
\begin{theorem}\label{thm.5}
Let $\Omega(a,A)=(b,{B})$; then for any $t>0$ and $T>0$ we have 
\[
\Omega(ta,{T} A)=\bigl(\tfrac{{T}}t b,{T}{B}\bigr).
\]
\end{theorem}
Next, we have some identities that link $(a,A)$ and $\Omega(a,A)$. 
\begin{theorem}\label{thm.6}
Let $(b,{B})=\Omega(a,A)$. Then the following identities hold true:
\begin{align}
\sum_{k=1}^NA_k&=\sum_{k=1}^N{B}_k,
\label{eq:12}
\\
\frac12\sum_{k=1}^N \frac{A_k}{a_k}&=\sum_{k=1}^N b_k,
\label{eq:13}
\\
\sum_{j,k=1}^N\frac{A_jA_k}{(a_j+a_k)^2}&=\sum_{k=1}^N b_k^2.
\label{eq:14}
\end{align}
\end{theorem}

The proofs of Theorems~\ref{thm.5} and \ref{thm.6} are given in Section~\ref{sec.f}.

\emph{Notation:}  $e_1,\dots,e_N$ is the standard basis in $\bbC^N$. 
We denote by $D_a$ the diagonal $N\times N$ matrix with elements $(a_1, \dots, a_N)$ on the 
diagonal.

\section{Proofs of the spectral properties of $\calC_{a,A}$ }\label{sec.b}

The following lemma is well-known; see e.g. \cite{Fiedler1}.
\begin{lemma}\label{lma.1}
The matrix $\calC_{a,A}$ is positive definite. 
\end{lemma}
\begin{proof}
We have
\begin{align*}
\jap{\calC_{a,A}x,x}
&=
\sum_{j,k=1}^N\sqrt{A_jA_k} \frac{x_j\overline{x_k}}{a_k+a_j}
\\
&=\sum_{j,k=1}^N\sqrt{A_jA_k} x_j\overline{x_k}\int_0^\infty e^{-a_jt}e^{-a_kt}\dd t
=
\int_0^\infty \Abs{\sum_{k=1}^N\sqrt{A_k}x_k e^{-a_kt}}^2\dd t\geq0.
\end{align*}
This proves that $\calC_{a,A}$ is positive \emph{semi}-definite. 
To prove that the kernel of $\calC_{a,A}$ is trivial, suppose that the above inner product 
vanishes. Then it is clear that the sum 
\[
\sum_{k=1}^N\sqrt{A_k}x_k e^{-a_kt}
\]
vanishes identically. Since all $a_k$ are distinct by assumption, it follows that $\sqrt{A_k}x_k=0$ 
for all $k$. Since all $A_k$ are positive, it follows that $x_k=0$ for all $k$. Thus, the kernel of 
$\calC_{a,A}$ is trivial. 
\end{proof}

It is a classical fact (see e.g. \cite{KS}) that Cauchy matrices satisfy a \emph{displacement 
equation} (known as the Lyapunov equation in control theory). 

\begin{lemma}\label{lma.2}
The displacement equation 
\begin{equation}
\calC_{a,A}D_a+D_a\calC_{a,A}=\jap{\bigcdot,v_A}v_A
\label{eq:1}
\end{equation}
holds true. Here $v_A$ is the vector \eqref{eq:19} and the right hand side denotes the rank one 
matrix $\{\sqrt{A_jA_k}\}_{j,k=1}^N$. 
\end{lemma}
\begin{proof}
By inspection, the matrix elements of the left hand side are 
\[
\frac{\sqrt{A_jA_k}}{a_k+a_j}a_j+a_k\frac{\sqrt{A_jA_k}}{a_k+a_j}=\sqrt{A_jA_k},
\]
as required.
\end{proof}

\begin{lemma}\label{lma.3}
The vector $v_A$ is a cyclic element of ${\calC}_{a,A}$, and therefore all eigenvalues of 
${\calC}_{a,A}$ are simple. 
\end{lemma}
\begin{proof}
For an eigenvalue $b$ of ${\calC}_{a,A}$ let $E_b:=\ker ({\calC}_{a,A} - b I)$ be the corresponding eigenspace and let  $P_b$ be the orthogonal projection onto $E_b$. 
Further, let $\sigma({\calC}_{a,A})$ be the spectrum of ${\calC}_{a,A}$ (i.e. the union of all eigenvalues). By the Spectral Theorem for Hermitian matrices 
\begin{align}\label{e:SpTh}
\bbC^N = \bigoplus_{b\in\sigma( {\calC}_{a,A} )} E_b, \qquad {\calC}_{a,A} = \sum_{b\in\sigma( 
{\calC}_{a,A} )} b P_b \ ,
\end{align}
so for any polynomial $p$
\begin{align*}
p({\calC}_{a,A})  = \sum_{b\in\sigma( {\calC}_{a,A}  )} p(b) P_b  .   
\end{align*}
Let $\operatorname{Pol}$ denote the set of polynomials.  Using the above  identity and Lagrange 
interpolation we get that
\begin{align*}
\Span\{({\calC}_{a,A})^n v_A: n\geq0\} & = \{p({\calC}_{a,A}) v_A: p \in \operatorname{Pol} \}
\\
& = \Span \{ P_b v_A: b\in\sigma \left(  {\calC}_{a,A}  \right)  \}  \\
& = \bigoplus_{b\in\sigma( 
{\calC}_{a,A} )} \Span\{P_b v_A\}  .  
\end{align*}

Assume for the sake of contradiction that $v_A$ is not cyclic. 
Then, comparing the above identity  to the first identity in \eqref{e:SpTh} and recalling that $P_b v_A\in E_b$,  we see that in this case  
$\Span\{P_b v_A\} \subsetneq E_b$ for some $b\in \sigma \left(  {\calC}_{a,A}  \right)$. Therefore there is a (non-zero) eigenvector in $f\in E_b$ orthogonal to $P_b v_A$.  Since the eigenspaces $E_b$ for different $b$ are orthogonal to each other, this means that $f\perp v_A$. 
Thus, there exist $f\ne0$ such that 
\[
{\calC}_{a,A}f=b f, \quad \jap{f,v_A}=0.
\]
By Lemma~\ref{lma.1}, we have $b>0$. Applying the displacement equation \eqref{eq:1} to the vector 
$f$, we find 
\[
{\calC}_{a,A}D_a f+D_a{\calC}_{a,A}f=0.
\]
Using the eigenvalue equation for $f$, this yields
\[
{\calC}_{a,A}D_a f=-b D_a f.
\]
Thus, $D_af$ is an eigenvector of ${\calC}_{a,A}$ corresponding to the eigenvalue $-b$. By 
Lemma~\ref{lma.1}, the operator ${\calC}_{a,A}$ has no negative eigenvalues, and so $D_af=0$. Since 
$a_k>0$ for all $k$, we find that $f=0$. This is a contradiction. 
\end{proof}

\begin{remark*}
Another cyclic element of ${\calC}_{a,A}$ that is easy to identify is 
\[
w_A=(\sqrt{A_1}/a_1,\dots,\sqrt{A_N}/a_N).
\]
Indeed, multiplying the displacement equation \eqref{eq:1} by $D_a^{-1}$ on the left and on the 
right, we find 
\[
\calC_{a,A}D_a^{-1}+D_a^{-1}\calC_{a,A}=\jap{\bigcdot,w_A}w_A.
\]
From here the cyclicity of $w_A$ follows by the same pattern. 
\end{remark*}

\section{Proof of the main result}\label{sec.e}

We first need a lemma about solutions to the displacement equation. This lemma is widely known in 
Lyapunov stability theory, cf \cite{Moore}.  An infinite dimensional version of this Lemma can be 
found in  \cite[Theorem 3.1]{MPT}.  We state it in a general form; however for our purposes  we 
will only need the uniqueness part. 

\begin{lemma}\label{lma.5}
Let $X$ and $K$ be $N\times N$ matrices such that $X$ is positive definite. Then there exists a 
unique $N\times N$ matrix $Y$ that solves the displacement equation 
\begin{equation}
XY+YX=K.
\label{*1}
\end{equation}
Moreover, $Y$ is given by the integral 
\begin{equation}
Y=\int_0^\infty e^{-tX}Ke^{-tX}\dd t. 
\label{*2}
\end{equation}
\end{lemma}
\begin{proof}
We first note that the integral in \eqref{*2} converges absolutely because $X$ is positive 
definite. Let $Y$ be given by the integral \eqref{*2} and let $F(t)=e^{-tX}Ke^{-tX}$. We have
\[
F'(t)=-e^{-tX}(XK+KX)e^{-tX}
\]
and therefore
\[
XY+YX=\int_0^\infty e^{-tX}(XK+KX)e^{-tX}\dd t=-\int_0^\infty F'(t)\dd t=F(0)=K.
\]
Thus, $Y$ is a solution to the displacement equation \eqref{*1}. To prove uniqueness, assume 
$\widetilde Y$ is a solution to \eqref{*1} and let $G(t)=e^{-tX}\widetilde{Y}e^{-tX}$. Then by a 
similar argument and using the displacement equation, we find
\begin{align*}
\widetilde Y&=G(0)=-\int_0^\infty G'(t)\dd t
=\int_0^\infty e^{-tX}(X\widetilde{Y}+\widetilde{Y}X)e^{-tX}\dd t
\\
&=\int_0^\infty e^{-tX}K e^{-tX}\dd t
=Y, 
\end{align*}
which proves uniqueness. 
\end{proof}

\begin{lemma}\label{lma.6}
Let $(b,{B})=\Omega(a,A)$. Then there exists an orthogonal $N\times N$ matrix $U$ such that:
\begin{equation}
Uv_A=v_B,
\quad
U{\calC}_{a,A}=D_bU, 
\quad
UD_a={\calC}_{b,{B}}U.
\label{eq:4}
\end{equation}
\end{lemma}
Before embarking on the proof, we make explicit our notation for clarity: $D_b$ is the $N\times N$ 
diagonal matrix with $(b_1,\dots,b_N)$ on the diagonal, ${\calC}_{b,{B}}$ is the Cauchy matrix 
\eqref{eq:17} with $(b,B)$ in place of $(a,A)$, and 
\[
v_B=(\sqrt{B_1},\dots,\sqrt{B_N}).
\]
\begin{proof}[Proof of Lemma~\ref{lma.6}]
Let $f_1,\dots,f_N$ be the normalised eigenvectors of ${\calC}_{a,A}$, corresponding to the 
eigenvalues $b_1,\dots,b_N$. Each vector $f_k$ is uniquely defined up to multiplication by a 
unimodular factor. Let us choose these factors such that 
\[
\jap{v_A,f_k}>0
\]
for all $k$. We recall that the inner product $\jap{v_A,f_k}$ is non-zero by Lemma~\ref{lma.3}. 

We define the map $U$ in $\bbC^N$ as
\[
Ux=(\jap{x,f_1},\dots,\jap{x,f_N})\in\bbC^N.
\]
In other words, $U$ is the matrix of the change of basis in $\bbC^N$ from $f_1,\dots,f_N$ to 
$e_1,\dots,e_N$. Since $\{f_k\}_{k=1}^N$ is an orthonormal basis, the matrix $U$ is orthogonal. 

Let us expand $v_A$ with respect to the basis $\{f_k\}_{k=1}^N$ and use the condition 
$\jap{v_A,f_k}>0$:
\[
v_A=\sum_{k=1}^N \jap{v_A,f_k} f_k=\sum_{k=1}^N \sqrt{B_k}f_k,
\]
according to our definition \eqref{eq:20} of $B_k$. Comparing this with the definition of $U$, we 
see that $Uv_A=v_B$, which is the first identity in \eqref{eq:4}. 

For any $x\in\bbC^N$ we have
\[
{\calC}_{a,A}x
= \sum_{k=1}^N c_k{\calC}_{a,A}f_k
=\sum_{k=1}^N c_kb_k f_k,
\quad\text{ where }\quad x=\sum_{k=1}^N c_kf_k
\]
and therefore 
\[
U{\calC}_{a,A}x=(c_1b_1,\dots,c_Nb_N).
\]
This can be written as
\[
U{\calC}_{a,A}=D_bU,
\]
and so we obtain the second identity in \eqref{eq:4}. 

The last identity in \eqref{eq:4} is less obvious. 
Let us write down the displacement equations \eqref{eq:1} for ${\calC}_{a,A}$ and 
${\calC}_{b,{B}}$: 
\begin{align}
{\calC}_{a,A}D_a+D_a{\calC}_{a,A}&=\jap{\bigcdot,v_A}v_A,
\label{eq:6}
\\
{\calC}_{b,{B}}D_b+D_b{\calC}_{b,{B}}&=\jap{\bigcdot,v_B}v_B.
\label{eq:7}
\end{align}
Let us multiply \eqref{eq:6} by $U$ on the left and by $U^*$ on the right:
\[
U{\calC}_{a,A}D_aU^*+UD_a{\calC}_{a,A}U^*=\jap{\bigcdot,Uv_A} Uv_A.
\]
Using the first two identities in \eqref{eq:4}, from here we find
\[
D_bUD_aU^*+UD_aU^*D_b=\jap{\bigcdot,v_B}v_B.
\]
Comparing this with \eqref{eq:7} and using the uniqueness part of Lemma~\ref{lma.5}, we find 
\begin{equation}
UD_aU^*={\calC}_{b,{B}},
\label{eq:8}
\end{equation}
which is exactly the last identity in \eqref{eq:4}. 
\end{proof}

\begin{proof}[Proof of Theorem~\ref{thm.4}] Let 
\[
(b,{B})=\Omega(a,A)\quad\text{ and }\quad (a',A')=\Omega(b,{B}).
\]
We need to check that $a'=a$ and $A'=A$. 
From \eqref{eq:8} we read off the eigenvalues of ${\calC}_{b,{B}}$: these are $a_1,\dots,a_N$. It 
follows that $a'=a$.

Let $U$ be the orthogonal matrix of Lemma~\ref{lma.6}. For $k=1,\dots,N$, denote $g_k=Ue_k$, where 
$e_1,\dots,e_N$ is the standard basis in $\bbC^N$. Using the last relation in \eqref{eq:4}, we find
\[
\calC_{b,B}g_k=(UD_aU^*)Ue_k=UD_ae_k=a_kUe_k=a_kg_k,
\]
and so $g_1,\dots,g_k$ is the orthonormal basis of eigenvectors of $\calC_{b,B}$. According to our 
definitions, 
\[
A'_k=\abs{\jap{v_B,g_k}}^2. 
\]
Using \eqref{eq:4}, we can rewrite this as
\[
A'_k=\abs{\jap{Uv_A,Ue_k}}^2=\abs{\jap{v_A,e_k}}^2=A_k.
\]
We have checked that $A'=A$. The proof is complete. 
\end{proof}

\section{Proofs of properties of $\Omega$}\label{sec.f}

\begin{proof}[Proof of Theorem~\ref{thm.5}]
Denote 
\[
\Omega(ta,{T} A)=(b',{B}').
\]
By the definition \eqref{eq:17} of ${\calC}_{a,A}$,  we have
\[
{\calC}_{ta,{T} A}=\frac{{T}}{t}{\calC}_{a,A}
\]
and therefore $b'=\frac{{T}}t b$. 
Let us find ${B}'$. By the above displayed equation, the eigenvectors $f_k$ of ${\calC}_{ta,{T} A}$ 
and ${\calC}_{a,A}$ coincide. We also have $v_{TA}=\sqrt{T}v_A$, and so 
\[
{B}'_k
=\abs{\jap{v_{TA},f_k}}^2
=T\abs{\jap{v_A,f_k}}^2
=TB_k.
\]
Thus, ${B}'={T}{B}$, as claimed. 
\end{proof}

\begin{proof}[Proof of Theorem~\ref{thm.6}]
We recall that by Lemma~\ref{lma.6} we have $v_B=Uv_A$, where $U$ is an orthogonal matrix. 
Hence $\norm{v_B}^2=\norm{v_A}^2$ (here $\norm{\cdot}$ is the standard Euclidean norm), which 
rewrites as the first identity \eqref{eq:12}. 

Let us write the second identity in \eqref{eq:4} as 
\begin{equation}
U{\calC}_{a,A}U^*=D_b.
\label{eq:15}
\end{equation}
Taking traces, we find
\[
\Tr{\calC}_{a,A}=\Tr D_b.
\]
Computing the traces yields \eqref{eq:13}.
Finally, from \eqref{eq:15} by squaring and taking traces we find
\[
\Tr{\calC}_{a,A}^2=\Tr D_b^2,
\]
and computing the traces on both sides yields \eqref{eq:14}.
\end{proof}

\section{Connection with results of \cite{PushTreil}}\label{sec.g}

\subsection{Pairs $(a,A)$ as measures}
It is illuminating to think of elements of the set $Z_N$ as measures. Indeed, let us associate with 
$(a,A)\in Z_N$ the measure $\alpha$ which consists of $N$ point masses at $a_1,\dots a_N$ with 
$\alpha(\{a_k\})=A_k$. Let us consider the $L^2$-space of functions on the real line with the 
measure $\alpha$; we denote it by $L^2(\alpha)$. Of course, the functions $f\in L^2(\alpha)$ are 
well-defined only at the points $a_1,\dots,a_N$, and so $L^2(\alpha)$ is isomorphic to $\bbC^N$. 
More explicitly, let us define the unitary map $V:L^2(\alpha)\to\bbC^N$ by 
\[
Vf=(\sqrt{A_1}f(a_1),\dots,\sqrt{A_N}f(a_N)).
\]
Then 
\[
G_\alpha=V^*{\calC}_{a,A}V
\]
is the operator in $L^2(\alpha)$, explicitly given by 
\begin{equation}
{G}_\alpha: f\mapsto \int_{\bbR}\frac{f(y)}{x+y}\dd\alpha(y). 
\label{eq:18}
\end{equation}
We see that $V^*v_A=\1$, i.e. the function identically equal to $1$ a.e. with respect to the 
measure $\alpha$. By Lemma~\ref{lma.3}, the function $\1$ is a cyclic element of $G_\alpha$. The 
spectral map $\Omega$ maps the measure $\alpha$ to another measure $\beta$ which is uniquely 
defined by the identity
\[
\jap{\varphi({G}_\alpha)\1,\1}_{\alpha}=\int_{\bbR}\varphi(x)\dd\beta(x)
\]
for all bounded continuous test-functions $\varphi$, where $\jap{\cdot,\cdot}_\alpha$ is the inner 
product in $L^2(\alpha)$. In other words, $\beta$ is the spectral measure of ${G}_\alpha$ 
corresponding to the vector $\1\in L^2(\alpha)$. Theorem~\ref{thm.4} says that the map 
$\alpha\mapsto\beta$ is an involution on the set of all measures that are finite linear 
combinations of point masses on the positive half-line $(0,\infty)$. 

In \cite{PushTreil}, operators $G_\alpha$ are studied for the class of all finite measures $\alpha$ 
on $[0,\infty)$ with $\alpha(\{0\})=0$. In this case, the operator $G_\alpha$ is in general 
unbounded. Nevertheless it is possible to define  $G_\alpha$ as a self-adjoint operator on 
$L^2(\alpha)$ and to define the measure $\beta$ as the spectral measure of $G_\alpha$ corresponding 
to the cyclic element $\1\in L^2(\alpha)$. This defines the spectral map $\alpha\mapsto\beta$ on 
finite measures. The main result of \cite{PushTreil} is that $\Omega$ is an involution on finite 
measures.

\subsection{Connection with Hankel operators}
The original motivation for this paper comes from the spectral theory of Hankel operators, studied 
in \cite{PushTreil}. We briefly explain the connection and refer to \cite{PushTreil} for the 
details (see also the important precursor \cite{MPT}). 

An integral Hankel operator in $L^2(\bbR_+)$ is the operator 
\begin{equation}
f\mapsto \int_0^\infty h(t+s)f(s)\dd s, 
\label{eq:16}
\end{equation}
where $h$ is known as the \emph{kernel function}. A Hankel operator is self-adjoint and positive 
semi-definite on $L^2(\bbR_+)$ if and only if the kernel function $h$ is the Laplace transform of a 
positive measure $\alpha$ on $\bbR_+$:
\[
h(t)=\int_0^\infty e^{-tx}\dd\alpha(x), \quad t>0.
\]
Let us denote by $\Gamma_\alpha$ the Hankel operator \eqref{eq:16} corresponding to a measure 
$\alpha$. Finite rank Hankel operators $\Gamma_\alpha$ correspond to measures $\alpha$ that are 
finite linear combinations of point masses on $(0,\infty)$. 

Let us explain the connection with the operators ${G}_\alpha$ as in \eqref{eq:18}. Let 
\[
L_\alpha: L^2(\alpha)\to L^2(\bbR)
\]
be the operator of the Laplace transform, 
\[
(L_\alpha f)(t)=\int_0^\infty f(x)e^{-tx}\dd\alpha(x).
\]
For general measures $\alpha$, the operator $L_\alpha$ is unbounded and its precise definition is a 
non-trivial question, discussed in \cite{PushTreil}. Of course, for $\alpha$ that are finite linear 
combinations of point masses, the space $L^2(\alpha)$ is finite-dimensional and so the definition 
of $L_\alpha$ is straightforward. 

We then observe that 
\[
\Gamma_\alpha=L_\alpha L_\alpha^*\quad\text{ and }\quad G_\alpha=L_\alpha^*L_\alpha. 
\]
It follows that the operators 
\[
\Gamma_\alpha|_{(\Ker\Gamma_\alpha)^\perp}\quad \text{ and }\quad G_\alpha
\]
are unitarily equivalent. This reduces the spectral analysis of the Hankel operator $\Gamma_\alpha$ 
to the analysis of the operator $G_\alpha$.

\end{document}